\documentclass[sn-mathphys-num]{sn-jnl}

\usepackage{graphicx}%
\usepackage{multirow}%
\usepackage{amsmath,amssymb,amsfonts}%
\usepackage{amsthm}%
\usepackage{mathrsfs}%
\usepackage[title]{appendix}%
\usepackage{xcolor}%
\usepackage{textcomp}%
\usepackage{manyfoot}%
\usepackage{booktabs}%
\usepackage{algorithm}%
\usepackage{algorithmicx}%
\usepackage{algpseudocode}%
\usepackage{listings}%

\theoremstyle{thmstyleone}%
\newtheorem{theorem}{Theorem}
\theoremstyle{thmstyletwo}%

\theoremstyle{thmstylethree}%
\newtheorem{corollary}{Corollary}%
\newtheorem{lemma}{Corollary}%
\begin{document}

\title[Suborbital graphs obtained by the modular congruence subgroup $\Gamma_{0}(L,M)$]{Suborbital graphs obtained by the modular congruence subgroup $\Gamma_{0}(L,M)$}


\author*[1]{\fnm{Ibrahim} \sur{G\u{o}kcan}}\email{gokcan@artvin.edu.tr}

\author[2]{\fnm{Ali Hikmet} \sur{De\u{g}er}}\email{ahikmetd@ktu.edu.tr}
\equalcont{These authors contributed equally to this work.}

\affil*[1]{\orgdiv{Faculty of Arts and Sciences}, \orgname{Artvin \c{C}oruh University},\city{Artvin}, \postcode{08000},\country{T\"{u}rkiye}}

\affil[2]{\orgdiv{Departmant of Mathematics}, \orgname{Karadeniz Technical University},\city{Trabzon}, \postcode{61000},\country{T\"{u}rkiye}}


\abstract{In the suborbital graphs studies, there has been a research gap in the sense that the Modular group is connected to two numbers. Thus, this paper attempts to contribute to the studies developed by Gauss, Bolyai, Lobachevsky and Riemann. However, this study mainly concentrates on the action of suborbital graphs obtained with the Modular congruence subgroup $\Gamma_{0}\left(L,M\right)$, making this study sui generis since it deals with the Modular group, connected to two numbers. In developing our graph action, we utilized the theories of non-Euclidean geometry. Investigating the congruence relation other than identity and universal relation, the number of congruence relation, transitive act on vertices and edges, edge condition for the congruence group $\Gamma_{0}\left(L,M\right)$, based on previously-obtained studies, we concluded with new theorems in this study. So, the results are obtained in this paper related to a different congruence modular subgroup provides various aspects of the same structure in mathematics and adapting it to such as algebraic geometry, number theory, differential geometry, topology and physics.
}




\maketitle
\section{Introduction}\label{sec1}
This paper analyzing the congruence relation other than identity and universal relation, the number of congruence relation, transitive act on vertices and edges, edge condition for the congruence group $\Gamma_{0}\left(L,M\right)$ consist of 4 sections.  In the first section, relevant literature about suborbital graphs and modular group $\Gamma$ is given. The second section gives the results related to obtained by act of the congruence group $\Gamma_{0}\left(L,M\right)$ on suborbital graphs $F_{u,L}$ and $F_{M,u}$. In the third section, since $\Gamma_{0}\left(L,M\right)<\Gamma_{0}\left(L\right)$ and $\Gamma_{0}\left(L,M\right)<\Gamma^{0}\left(M\right)$, it is highlighted that $\Gamma_{0}\left(L,M\right)$ generates different orbitals according to $\infty$ and $0$. The fourth section briefly gives the results obtained in this paper.

Euclid's Elements has been translated into many different languages and studied for centuries. The possibility of situations where the 5th postulate in Euclid?s Elements did not hold took mathematical thinking forward and led to the discovery of non-Euclidean geometries. For this reason,the 19th century was a milestone for mathematics. Gauss, Bolyai, Lobachevsky, and Bernhard Riemann studied non-Euclidean geometries. The perspective of mathematics, a joint product of many mathematicians, was broadened in the 19th century as a result of the studies on non-Euclidean geometries by the mathematician Bernhard Riemann. Elliptic and Hyperbolic Geometries are examples of non-Euclidean geometries. In previous years, many topics that had been part of mathematics began to be studied in these geometries. Graph theory is an example of these topics. By transferring Graph theory to Elliptic and Hyperbolic geometries, studies of these fields were initiated. In the 19th century, Henry Poincare presented some results for the theory of discrete groups, and the results gave rise to the generalization of the theory of elliptic functions. Invariant theory led to further studies on linear fractional congruence groups. In \cite{1}, some ideas have been put forward about graph action. These ideas can be summarized as follows:

Assume that $\left(G,\Omega\right)$ is a transitive permutation group. So, $G$ acts on $\Omega\times \Omega$ with transform $g:\left(\alpha,\beta\right) \rightarrow\left(g\left(\alpha\right),g\left(\beta\right)\right)$, where $g\in G$ and $\alpha,\beta\in \Omega$. The orbitals of this action is called the suborbitals of $G$. The orbital that contains $\left(\alpha,\beta\right)$ is demonstrated with $O\left(\alpha,\beta\right)$ and suborbital graph $G\left(\alpha,\beta\right)$, which is the constitute of suborbital $O\left(\alpha,\beta\right)$. In this suborbital graph, the vertices are considered as the elements on  $\Omega$ and the edge is a directed path which goes from vertex $\gamma$ to vertex $\delta$. In addition, $O\left(\beta,\alpha\right)$ is an orbital, too. $O(\beta,\alpha)$ is equal to $O\left(\alpha,\beta\right)$ or it is not. If it is equal, it is $G\left(\alpha,\beta\right)=G\left(\beta,\alpha\right)$ and this suborbital graph has pair directed edges. If it is not, $G\left(\alpha,\beta\right)$ has inversed directed edges, according to $G\left(\beta,\alpha\right)$.

$G$ acts on $\Omega$ transitively. ''$\approx$'' is called as the congruence relation on $\Omega$ whenever $g\left(\alpha\right)\approx g\left(\beta\right)$ for $\alpha\approx \beta$, where $g\in G$ and $\alpha,\beta \in\Omega$. This congruence relation is called $G$-invariant congruence relation. It is called blocks for the congruence classes of this relation. There are always two congruence relations on $\Omega$. These are identity and universal congruence relations. Identity congruence relation can be demonstrated with $\alpha\approx \beta \Rightarrow \alpha=\beta$ and universal congruence relation can be demonstrated $\alpha\approx \beta$, for $\alpha,\beta \in\Omega$. If there is a congruence relation different from identity and universal congruence relations, the action of $G$ on $\Omega$ is called imprimitive, if there is not, we name it primitive. The imprimitive action has to be transitive.

In \cite{2}, the authors proved that $\left(G,\Omega\right)$ is an primitive permutation group provided that $G_{\alpha}$ is the stabilizer of vertex $\alpha \in \Omega$ which is a maximal subgroup of  $G$ for $\forall \alpha \in \Omega$. Furthermore, $\left(G,\Omega\right)$ is an imprimitive permutation group on the condition that there is a subgroup $H$ of $G$ providing $G_{\alpha}\lneqq H \lvertneqq G$.

Modular group and its subgroups, which play an important role in Fermat's last theorem, are extensively studied in the field of Graph theory. Hyperbolic geometry, one of non-Euclidean geometries, is the focus of this study. In this geometry, the upper plane $\mathbb{H}$ is demonstrated with set $\lbrace z\in \mathbb{C}: Imz>0\rbrace$.If Modular group is demonstrated with $\Gamma$, it is 
$$
\Gamma=PSL(2,\mathbb{Z})\cong SL(2,\mathbb{Z})/ \lbrace \mp I \rbrace.
$$
Modular group $\Gamma$ is given as follows:
\begin{equation}
\Gamma=\left\lbrace \left(\begin{array}{ccc} a & b \\ c & d \end{array} \right)\mid a,b,c,d \in\mathbb{Z},ad-bc=1 \right\rbrace
\end{equation}
The elements of Modular group $\Gamma$ are formulated with M\"{o}bius transforms. That is
\begin{equation}
T\left(z\right)=\frac{az+b}{cz+d}, T\in \Gamma.
\end{equation}
In addition to this, there are numerous congruence subgroups of the modular group $\Gamma$. Some of these can be listed as follows:
\begin{equation}
\Gamma_{1}\left( N \right)=\left\lbrace \left( \begin{array}{ccc} a & b \\ c & d \end{array} \right)\in\Gamma\mid a \equiv d \equiv 1 \left( mod N \right), b \equiv c \equiv 0 \left( mod N\right) \right\rbrace
\end{equation}
\begin{equation}
\Gamma_{0}(N)=\left\lbrace \left( \begin{array}{ccc} a & b \\ c & d \end{array} \right)\in\Gamma\mid c \equiv 0 \left( mod N \right) \right\rbrace
\end{equation}
\begin{equation}
\Gamma^{0}(N)=\left\lbrace \left( \begin{array}{ccc} a & b \\ c & d \end{array} \right)\in\Gamma\mid b \equiv 0 \left( mod N \right) \right\rbrace
\end{equation}
\begin{equation}
\Gamma_{0}^{0}\left(N\right)=\left\lbrace \left( \begin{array}{ccc} a & b \\ c & d \end{array} \right)\in\Gamma\mid b \equiv c \equiv 0 \left( mod N \right) \right\rbrace
\end{equation}
The congruence subgroups given in Equations (3)-(6) have intensely been studied in recent years. Especially, Equation (4) has been studied in \cite{2,3,4,5} and Equation (5) has been studied in \cite{6,7}.

Based on these ideas, it has been studied the Modular groups in case of $G=\Gamma$ and $\Omega=\widehat{\mathbb{Q}}$ in \cite{2}. In addition to, $\Gamma_{\infty}$ ,which is stabilizer of $\infty$, is a subgroup of $\Gamma$ generated by matrix $\left( \begin{array}{ccc} 1 & 1 \\ 0 & 1 \end{array} \right)$. Then, the invariant congruence equations $\Gamma$ on $\widehat{\mathbb{Q}}$ could be obtained should be there is a subgroup $U$ of $\Gamma$, which contains $\Gamma_{\infty}$ or equally $\left( \begin{array}{ccc} 1 & 1 \\ 0 & 1 \end{array} \right)$. So, subgroup $U$ can be accepted as $\Gamma_{0}(N)$. In other words, $\Gamma_{\infty}<\Gamma_{0}(N)\leq \Gamma$ for $N\in \mathbb{N}$ and $\Gamma_{\infty}<\Gamma_{0}(N)<\Gamma$. Therefore, $\Gamma$ acts imprimitively on $\widehat{\mathbb{Q}}$. 

Assume that invariant congruence equation is $\approx_{N}$. Then, $g\left(\infty\right)\approx_{N} g^{'}\left(\infty\right)\Leftrightarrow \infty \approx_{N} g^{-1}g^{'}\left(\infty\right) \Leftrightarrow g^{-1}g^{'} \in U=\Gamma_{0}\left(N\right)$ in which $g=\left( \begin{array}{ccc} r & * \\ s & * \end{array} \right), g^{'}=\left( \begin{array}{ccc} x & * \\ y & * \end{array} \right)\in \Gamma$ and $v=\frac{r}{s}, w=\frac{x}{y}\in \widehat{\mathbb{Q}}$. That is, $v \approx_{N} w \Leftrightarrow g^{-1}g^{'} \in \Gamma_{0}\left(N \right)$. 
$$
g^{-1}g^{'}=\left( \begin{array}{ccc} * & -s \\ * & r \end{array} \right)
\left( \begin{array}{ccc} x & * \\ y & * \end{array} \right)=\left( \begin{array}{ccc} * & * \\ ry-sx & * \end{array} \right)\in\Gamma_{0}\left(N\right).$$ 
Consequently, $v \approx_{N} w \Leftrightarrow ry-sx\equiv 0 \left(mod N\right)$ can be found. So, $x\equiv ur \left(mod N\right)$ and $y\equiv us \left(mod N\right)$ can be obtained for $u\in N$. From here, the number of invariant congruence equation is given as follows
$$
\varphi\left(N\right)=N\prod_{p/N}\left(1+\frac{1}{p}\right)
$$
for $p$ prime numbers in \cite{2}. 

However, they investigated the suborbital graphs $G_{u,N}$ and $F_{u,N}$, which obtained an element of the Modular group $\Gamma$, for $\left(u,N\right)=1$ and $u<N$. Also, they provided edge conditions for $G_{u,N}$ and $F_{u,N}$. Edge conditions for $F_{u,N}$ and $G_{u,N}$ are presented in \cite{2}, where $\frac{r}{s}$ and $\frac{x}{y}$ are vertices on $F_{u,N}$ and $G_{u,N}$:\\
There is an edge from vertex $\frac{r}{s}$ to vertex $\frac{x}{y}$ provided that if $x\equiv\mp ur \left(mod N\right)$ or $y\equiv\mp us \left(mod N\right)$; $ry-sx=\mp N$.
In addition, the state of being a forest in the graph has been given as a conjecture. The conjecture has been solved in \cite{3}.\\
In \cite{4}, author studied Modular groups and its subgroups comprehensively. Moreover, the author presented some conclusions that we also used in this article.

\begin{lemma}\cite{4} Assume that $T\left(N \right)$ is a multiplication function for $\left(N_{1},N_{2}\right)=1$. Thus, $T\left(N_{1}N_{2}\right)=T\left(N_{1}\right)T\left(N_{2}\right)$.
\end{lemma}
\begin{theorem} \cite{4} Assume that $N_{1}$ and $N_{2}$ are positive integers and $\lbrace N_{1},N_{2}\rbrace$ is the least common multiple of these positive integers. Then,
\item[a.] $\Gamma\left(N_{1}\right)\cap \Gamma\left(N_{2}\right)=\Gamma\left(\lbrace N_{1},N_{2}\rbrace\right)$
\item[b.] $\Gamma\left(N_{1}\right)\Gamma\left(N_{2}\right)=\Gamma\left(\left( N_{1},N_{2}\right)\right)$
\end{theorem}
\begin{theorem} \cite{4} Assume that $N_{1}$ and $N_{2}$ are positive integers. So,
\item[a.]$\Gamma\left(N_{1}\right)\Gamma^{0}\left(N_{2}\right)=\Gamma^{0}\left(\left( N_{1},N_{2}\right)\right)$
\item[b.]$\Gamma\left(N_{1}\right)\Gamma_{0}\left(N_{2}\right)=\Gamma_{0}\left(\left( N_{1},N_{2}\right)\right)$.
\end{theorem}
These notations presented in \cite{4} have not been confused with the notations of the present study, as the presumptions of these studies are different. 

However, the authors studied the action of  $\Gamma_{0}\left(p\right)$ on $\widehat{\mathbb{Q}}$ in which $p$ is a prime number in \cite{5}. In this article, they used orbit
$$
\left(\begin{array}{ccc} a\\b \end{array}\right)=\left\lbrace \frac{x}{y}\in\widehat{\mathbb{Q}}:(p,y)=b, x\equiv a\mod\left(b,\frac{p}{b}\right)\right\rbrace.
$$
Since the action on $\widehat{\mathbb{Q}}$ is transitive, they utilized the orbit
$$
\left(\begin{array}{ccc} 1\\p \end{array}\right)=\left\lbrace \frac{x}{y}\in\widehat{\mathbb{Q}}:(p,y)=b\right\rbrace
$$
that contains $\infty$.

In \cite{6}, the authors examined the action of  $\Gamma^{0}\left( p \right)$ and they presented the edge conditions for $p$ prime numbers. Here, they discussed the orbit
$$
\left(\begin{array}{ccc} a\\b \end{array}\right)=\left\lbrace \frac{x}{y}\in\widehat{\mathbb{Q}}:(p,y)=b, x\equiv a\mod(b, \frac{p}{b})\right\rbrace.
$$
They also demonstrated that the action of  $\Gamma^{0}\left( p \right)$ is not transitive, presenting the group:
$$
\Gamma^{*}(p)=\left\lbrace \left( \begin{array}{ccc} 1+ap & bp \\ c & 1+dp \end{array} \right)\mid bp\equiv 0\left(\mod p\right)\right\rbrace
$$
The group $\Gamma^{*}(p)$ is provided order $\Gamma^{0}<\Gamma^{*}(p)<\Gamma^{0}(p)$. Furthermore, the orbit
$$
\left(\begin{array}{ccc} 1\\p \end{array}\right)=\left\lbrace \frac{x}{y}\in\widehat{\mathbb{Q}}:(p,y)=b\right\rbrace
$$
that contains $0$ is chosen to make the act transitive. In addition, they indicated that $\Gamma^{0}\left(N\right)$ acts imprimitively on the orbit $\left(\begin{array}{ccc} 1\\p \end{array}\right)$. Furthermore, the authors investigated the suborbital graphs obtained by $\Gamma^{0}\left(N\right)$. Since all the blocs acquired by $\Gamma^{0}\left(N\right)$ are isomorphic, they use only block $\left[0\right]$. Hence, they presented edge conditions for suborbital graphs $F_{p,q}$ as follows:

There is an edge from vertex $\frac{r}{s}$ to vertex $\frac{x}{y}$ if and only if $x\equiv \mp pr\left(\mod q\right),r\equiv 0 \left(\mod q\right)$ or $y\equiv \mp ps  \left( \mod q\right),s\equiv 1 \left( \mod q\right);ry-sx=\mp q$.

Also, the action of  $\Gamma^{0}\left(N\right)$ and edge conditions on $\widehat{\mathbb{Q}}$, $G_{p,q}$ and $F_{p,q}$ are explored in \cite{7}. In this paper, the orbit
$$
\left(\begin{array}{ccc} 1\\1 \end{array}\right)=\left\lbrace \frac{x}{y}\in\widehat{\mathbb{Q}}:(N,x)=1 \right\rbrace
$$
is investigated in which $N$ is a natural number. Moreover, the authors presented the set of orbits $\Gamma^{0}\left(N\right)$ and expressed that $\left(\Gamma^{0}\left(N\right),\left(\begin{array}{ccc} 1\\1 \end{array}\right)\right)$ is a transitive permutation group. Here, the edge conditions are provided as follows:

There is an edge from vertex $\frac{r}{s}$ to vertex $\frac{x}{y}$ if and only if $x\equiv \mp pr\left(\mod Nq\right),r\equiv 0 \left(\mod q\right)$ or $y\equiv \mp ps  \left(\mod q\right),s\equiv 1 \left(\mod q\right);ry-sx=\mp q$.

In \cite{8}, the relation has been investigated between suborbital graphs and continued fractions, and a vertex of a suborbital graph has been denoted as a continued fraction. In \cite{9}, the vertices obtained as a continued fraction have been expressed with Fibonacci numbers, and the vertices  obtained have been generalized. Also, the author defined minimal length between vertices and presented a relevant theorem for minimal length in \cite{10}. In \cite{11}, the author has worked on defining the Modular group and its congruence subgroups, which have been widely discussed. The suborbital graphs were studied for the Atkin-Lehner group in \cite{12}.

The aim of this article is to seek for an answer to given as a proposition in \cite{13}. In this article, with the help of the information presented above, the congruence relation other than identity and universal relation, transitive act on vertices and edges, edge condition are represented about that are the congruence subgroups of $\Gamma_{0}\left(L,M\right)$. The congruence subgroups of $\Gamma_{0}\left(L,M\right)$ and $\Gamma_{0}^{0}\left(L,M\right)$ can be presented in which $L$ and $M$ are integers as follows:

\begin{multline}
\Gamma_{0}\left( L,M \right) =\Bigg\{ \left( \begin{array}{ccc} a & b \\ c & d \end{array} \right)\in\Gamma\mid a \equiv 1 \left(\mod L\right),d \equiv 1 \left(\mod M\right), \\
c \equiv 0 \left(\mod L \right),b \equiv 0 \left(\mod M \right)\Bigg\}
\end{multline}

\begin{equation}
\Gamma_{0}^{0}\left( L,M \right)=\left\lbrace \left( \begin{array}{ccc} a & b \\ c & d \end{array} \right)\in\Gamma\mid c \equiv 0 \left(\mod L \right), b \equiv 0 \left(\mod M \right) \right\rbrace
\end{equation}
Even though they are presented through exact theoretical mathematics, such as modular groups and suborbital graph theory, conclusions of this research have strong practical applicability to numerous areas of engineering. The properties of suborbital graphs such as symmetry, transitivity, and edge qualities as determined by congruence subgroups directly relate to the fundamental principles of building effective, robust, and modular engineering systems.

The uniform response of these graphs to group actions makes them useful in areas of electrical and computer science, where they enable understanding of network structure, encoding and decoding, and cryptosystem design. In mechanical and structural science, invariance to transformation and modularity properties of these graphs finds practical application in the analysis of dynamic and periodic structures. More generally, structures mathematically defined by the action of $\Gamma_{0}\left( L,M \right)$ provide an exact suite of principles that can be generalized to successfully model behaviors of real-system examples of discrete, symmetric, or hierarchical interactions.

The promise of interdisciplinary uses highlights the ability of theoretical mathematics to provide new insights into solving engineering problems, particularly in systems that benefit from greater accuracy, consistency, and stronger structural integrity.

Throughout the study, + demonstrates that $\frac{r}{s}$ is greater than $\frac{x}{y}$ and $-$ demonstrates that $\frac{x}{y}$ is greater than $\frac{r}{s}$. The proofs in the study are made for $+$. The proofs can however be reached in the same way for $-$.
\section{Results}\label{sec2}
\subsection{Some Conclusions Obtained by $\Gamma_{0} \left(L,M\right)$ on $F_{u,L}$}
In this section, we aim at drawing some conclusions about the suborbital graphs for $G=\Gamma_{0} \left(L,M \right)$ and $\Omega=\widehat{\mathbb{Q}}$. Firstly, we have to indicate that $\Gamma_{0} \left(L,M \right)$ acts on $\widehat{\mathbb{Q}}$ transitively. From \cite{6}, we have to find a maximal subgroup of $\Gamma_{0} \left(L,M \right)$, which $\Gamma_{0} \left(L,M \right)$ acts transitively on $\widehat{\mathbb{Q}}$. 

$\Gamma_{0}\left( L,M \right)$ acts transitively on orbit $\left(\begin{array}{ccc} 1\\L \end{array}\right)$. So, $\left(\Gamma_{0}\left( L,M \right),\widehat{\mathbb{Q}} \right)$ has to be imprimitive permutative group. Then, there is a congruence relation different from identity and universal relations. Assume that two elements are $g= \left( \begin{array}{ccc} a_{1} & b_{1} \\ c_{1} & d_{1} \end{array} \right)$ and $g^{'}=\left( \begin{array}{ccc} a_{2} & b_{2} \\ c_{2} & d_{2} \end{array} \right)$ in $\Gamma_{0}\left( L,M \right)$ for $\Gamma_{0,\infty}\left(L,M\right)<\Gamma_{0}\left(L,M\right)<\Gamma$, where $\Gamma_{0,\infty}\left(L,M\right)$ is a stabilizer of $\infty$ . From here, it can be presented as $g\left(\infty\right)\approx g^{'}\left(\infty\right)\Leftrightarrow \infty \approx g^{-1}g^{'}\left(\infty\right) \Leftrightarrow g^{-1}g^{'} \in \Gamma_{0}\left( L,M \right)$. That is,
$$
g^{-1}g^{'}=\left( \begin{array}{ccc} d_{1} & -b_{1} \\ -c_{1} & a_{1} \end{array} \right)\left( \begin{array}{ccc} a_{2} & b_{2} \\ c_{2} & d_{2} \end{array} \right)$$
$$=\left( \begin{array}{ccc} d_{1} a_{2} -b_{1} c_{2}  & d_{1} b_{2} -b_{1} d_{2}  \\ -a_{2} c_{1} +a_{1} c_{2}  & a_{1} d_{2} -c_{1} b_{2}\end{array} \right)
$$
$$
=\left( \begin{array}{ccc} d_{1} a_{2}  & d_{1} b_{2} -b_{1} d_{2}\\  a_{1} c_{2} -a_{2} c_{1} & a_{1} d_{2} \end{array} \right)\in \Gamma_{0}\left( L,M \right)
$$
for $g^{-1}=\left( \begin{array}{ccc} d_{1} & -b_{1} \\ -c_{1} & a_{1} \end{array} \right)$, $ a_{1} c_{2} -a_{2} c_{1}\equiv 0 \left(\mod L \right)$ and $ d_{1} b_{2} -b_{1} d_{2}\equiv 0 \left(\mod M \right)$. For $\Gamma_{0,\infty}\left(L,M\right)<\Gamma_{0}\left(L,M\right)<\Gamma$, it is obtained $d_{1} a_{2}\equiv 1 \left(\mod L \right)$ and $a_{1} d_{2}\equiv 1 \left(\mod M \right)$. $d_{1}a_{2} a_{1}\equiv a_{1} \left(\mod L \right)$ is obtained from $d_{1} a_{2}\equiv 1 \left(\mod L \right)$. Since $a_{1} a_{2}\equiv 1 \left(\mod L \right)$, $d_{1}\equiv a_{1} \left(\mod L \right)$. Similarly, $a_{1}d_{2} d_{1}\equiv d_{1} \left(\mod M \right)$ is found from $a_{1} d_{2}\equiv 1 \left(\mod M \right)$. Then, $a_{1}\equiv d_{1} \left(\mod M \right)$ for $d_{1} d_{2}\equiv 1 \left(\mod M \right)$. So, $g^{-1}g^{'} \in \Gamma_{0}\left(L,M \right)$ can be accepted.

In addition, if $\frac{x}{y}$ is reduced on $\widehat{\mathbb{Q}}$, the representation, under $\Gamma_{0}\left( L,M \right)$, is also reduced. So,
$$
\left( \begin{array}{ccc} a & b \\ c & d \end{array} \right)\frac{x}{y}=\frac{ax+by}{cx+dy}
$$
$$
c \left(ax+by \right)-a\left(cx+dy\right)=cax+cby-acx-ady=\left(cb-ad\right)y=-y
$$ 
$$
d\left(ax+by\right)-b\left(cx+dy\right)=dax+dby-bcx-bdy=x\left(da-bc\right)=x.
$$
Then, $\left(ax+by,cx+dy\right)=1$.

\begin{theorem}
Assume that $p_{i}$ are prime factors for $L$, $r$ is a multiplier other than prime numbers and $q_{j}$ are prime factors for $M$, $s$ is a multiplier other than prime numbers in which $i=1,2,\cdots,n$, $j=1,2,\cdots,m$.
$$
\varphi\left(L,M\right)=r\prod_{i=1}^{n}\left(p_{i}+1\right)+s\prod_{j=1}^{m}\left(q_{j}+1\right)
$$
\end{theorem}
\begin{proof}
In $\Gamma_{0}\left( L,M \right)$, there are two invariant congruence relations. Because, $\Gamma_{0}$, which is stabilizer of $0$ and $\Gamma_{\infty}$ are subgroups of $\Gamma_{0}\left( L,M \right)$. So, we have to obtain values both $L$ and $M$. 

For $L$ with given above,
$$
\varphi\left(L\right)=L\prod_{p_{i}/L}\left(1+\frac{1}{p_{i}}\right).
$$
Similarly in [2, Lemma 2.2.], we have to demonstrate every two sides of equation to be multiplicative functions. Assume that $L=\omega_{1}\omega_{2}$, for $\left(\omega_{1},\omega_{2}\right)=1$. From the invariant congruence relation $\approx_{L}$, $v_{L}\approx_{L}w_{L}\Leftrightarrow v_{L}\approx_{\omega_{1}}w_{L}$ and $v_{L}\approx_{\omega_{2}}w_{L}$. On the other hand, it can be written  $L=rp_{1}p_{2}\cdots p_{n}$. So, $L$ is obtained as multiplier of $p$ prime numbers. 

If it is taken as $v=\frac{a_{1}}{c_{1}}$, one of the $a_{1}$ and $c_{1}$ is relatively prime with $p_{i}$ numbers. Then,
$$
\varphi\left(L\right)=rp_{1}p_{2}\cdots p_{n}\left(1+\frac{1}{p_{1}}\right)\left(1+\frac{1}{p_{2}}\right)\cdots \left(1+\frac{1}{p_{n}}\right)
$$
$$
=rp_{1}p_{2}\cdots p_{n}\left(\frac{1+p_{1}}{p_{1}}\right)\left(\frac{1+p_{2}}{p_{2}}\right)\cdots \left(\frac{1+p_{n}}{p_{n}}\right)
$$   
$$
=r\left(1+p_{1}\right)\left(1+p_{2}\right)\cdots \left(1+p_{n}\right)
$$
$$
=r\prod_{i=1}^{n}\left(1+p_{i}\right)
$$
Similar operations can be made for $M$. Hence, it is found 
$$
\varphi\left(M\right)=s\prod_{j=1}^{m}\left(1+q_{j}\right).
$$
Since $\Gamma_{0}\left( L,M \right)$ has two invariant congruence relation, the number of invariant congruence relation is the sum of $\varphi\left(L\right)$ and $\varphi\left(M\right)$. Consequently,
$$
\varphi\left(L,M\right)=r\prod_{i=1}^{n}\left(1+p_{i}\right)+s\prod_{j=1}^{m}\left(1+q_{j}\right).
$$
\end{proof}
At the moment, we will present a corollary about the transitive action obtained by $\Gamma_{0}\left( L,M \right)$ on $\widehat{\mathbb{Q}}$.
\begin{corollary}
\item[a.] Action obtained by $\Gamma_{0}\left( L,M \right)$ on $\widehat{\mathbb{Q}}$ is transitive.
\item[b.] The stabilizer of each vertex of $\widehat{\mathbb{Q}}$ is infinite period.
\end{corollary}
\begin{proof}
\textbf{a.} For this, we have to show that there is a $T_{0}\in \Gamma_{0}\left(L,M\right)$ providing $T_{0}\left(\frac{a}{c}\right)=\left(\frac{b}{d}\right)$, where $\frac{a}{c}$ and $\frac{b}{d}$ are vertices on $\widehat{\mathbb{Q}}$. Assume that $\gamma\left(\infty\right)=\frac{a}{c}$ and $\mu \left(\infty\right)=\frac{b}{d}$ for $\gamma,\mu\in \Gamma_{0}\left( L,M \right)$. Then,$\infty=\gamma^{-1}\left(\frac{a}{c}\right)$. From here, $\mu \left(\gamma^{-1}\left(\frac{a}{c}\right)\right)=\frac{b}{d}$. So, $\mu \gamma^{-1}\left(\frac{a}{c}\right)=\frac{b}{d}$. It can be considered as $T_{0}=\mu \gamma^{-1}$.

It can be seen that vertex $\frac{a}{c}$ is on the orbital of $\infty$ for $\left(a,c\right)=1$. So,
$\left( \begin{array}{ccc} a & m \\ c & n \end{array} \right)\infty=\left(\begin{array}{ccc} a\\c \end{array}\right)$ for $\left( \begin{array}{ccc} a & m \\ c & n \end{array} \right)\in \Gamma_{0}\left(L,M\right)$.

\textbf{b.} Here, we have to find that there is a $T_{0}\in \Gamma_{0}\left(L,M\right)$ providing that $T_{0}\left(\infty\right)=\infty$. Then, $\left( \begin{array}{ccc} a & b \\ c & d \end{array} \right)\left(\begin{array}{ccc} 1\\0 \end{array}\right)=\left(\begin{array}{ccc} a\\c \end{array}\right)=\left(\begin{array}{ccc} 1\\0 \end{array}\right)$.
From here, $a=1$ and $c=0$. From the determination of $T_{0}$, $ad-bc=ad=1$ is found. This equation led to the following values: $a=1$ is $d=1$. In addition, b has an infinite value in according to $\mod L$, for $\infty$. So, $T^{0}$ can be written as follows:
$$
T^{0}=\left( \begin{array}{ccc} 1+aL & b \\ c & 1+dM \end{array} \right)=\left\lbrace \left( \begin{array}{ccc} 1 & b \\ 0 & 1 \end{array} \right), b\in \mathbb{Z} \right\rbrace=\left\lbrace \left\langle \begin{array}{ccc} 1 & b \\ 0 & 1 \end{array} \right\rangle \right\rbrace
.$$
Similarly, in this proof, if it is taken $\frac{0}{1}$ as vertex, $T_{0}$ can be found as
$$
T_{0}=\left( \begin{array}{ccc} 1+aL & b \\ c & 1+dM \end{array} \right)=\left\lbrace \left( \begin{array}{ccc} 1 & 0 \\ c & 1 \end{array} \right),c\in \mathbb{Z} \right\rbrace=\left\lbrace \left\langle \begin{array}{ccc} 1 & 0 \\ c & 1 \end{array} \right\rangle \right\rbrace
.$$ 
So, $c$ has infinite values in according to $\mod M$.
\end{proof}
\begin{theorem}
Assume that $\frac{r}{s}$ and $\frac{x}{y}$ are vertices on $F_{u,L}$. Then, there should be an edge from the vertex $\frac{r}{s}$ to the vertex $\frac{x}{y}$, provided that $x\equiv \mp ur\left(\mod L\right)$ and $y\equiv \mp us  \left(\mod L\right)$; $ry-sx=\mp L$.
\end{theorem}
\begin{proof}
 Assume that there is an edge from $\frac{r}{s}$ to $\frac{x}{y}$ in $F_{u,L}$. We have to find an element of $\Gamma_{0}\left(L,M\right)$, transferring from $\left( \begin{array}{ccc} 1 & u \\ 0 & L \end{array} \right)$ to $\left( \begin{array}{ccc} r & x \\ s & y \end{array} \right)$. Therefore, matrix $T_{0}=\left( \begin{array}{ccc} 1+aL & b \\ c & 1+dM \end{array} \right)$ can be selected as an element of $\Gamma_{0}\left(L,M\right)$. Thus, the transform is defined as
$$
\left( \begin{array}{ccc} 1+aL & b \\ c & 1+dM \end{array} \right)\left( \begin{array}{ccc} 1 & u \\ 0 & L \end{array} \right)=\left( \begin{array}{ccc} r & x \\ s & y \end{array} \right).
$$
At the same time, $T_{0}\left(\frac{1}{0}\right)=\frac{r}{s}$ and $T_{0}\left(\frac{u}{L}\right)=\frac{x}{y}$ is provided. Hence,  $r\equiv 1\left(\mod L\right)$ and  $s\equiv 0 \left(\mod L\right)$ are reached. Moreover, the matrix multiplication
$$
\left( \begin{array}{ccc} 1+aL & u(1+aL)+bL \\ c & uc+(1+dM)L \end{array} \right)=\left( \begin{array}{ccc} r & x \\ s & y \end{array} \right)
$$
is obtained. $r=1+aL$ and $s=c$ can be directly seen. Moreover, 
$x\equiv \mp ur\left(mod L\right)$ is obtained for equation $x=u\left(1+aL\right)+bL=ur+bL$ and $y\equiv \mp us  \left( mod L\right)$ is obtained for equation $y=uc+\left(1+dM\right)L=us+\left(1+dM\right)L$. $ry-sx=L$ can be easily seen.

By contrast, assume that $r\equiv 1\left(\mod L\right)$, $s\equiv 0 \left(\mod L\right)$, $x\equiv \mp ur\left(\mod L\right)$ and $y\equiv \mp us  \left(\mod L\right)$ are provided. We have to demonstrate that there is an edge. Hence, $x=ur+aL$ is obtained from $x\equiv \mp ur\left(\mod L\right)$  and $y=us+cL$ is obtained from $y\equiv \mp us  \left(\mod L\right)$, for $a,c\in \mathbb{Z}$. Moreover, we can write the following equation:
$$
\left( \begin{array}{ccc} r & a \\ s & c \end{array} \right)\left( \begin{array}{ccc} 1 & u \\ 0 & L \end{array} \right)=\left( \begin{array}{ccc} r & ur+aL\\ s & us+cL \end{array} \right)
$$
Since $r\equiv 1\left(\mod L\right)$ and $s\equiv 0 \left(\mod L\right)$, $c\equiv 1\left(\mod L\right)$ is obtained from $rc-sa=1$.
Consequently, $\left( \begin{array}{ccc} r & a \\ s & c \end{array} \right)$ is an element of $\Gamma_{0}\left(L,M\right)$.
\end{proof}

\begin{theorem}
$\Gamma_{0}\left(L,M\right)$ transitively permutes the vertices and the edges on $F_{u,L}$.
\end{theorem}
\begin{proof}
 At first, let us present that $\Gamma_{0}\left(L,M\right)$ permutes the vertices. Here, assume that the vertices $\frac{u_{1}}{L}$ and $\frac{u_{2}}{L}$ are on orbit $\left(\begin{array}{ccc} 1\\0 \end{array}\right)$  and $\approx$ is an invariant equivalence relation. Due to transitive action, there is an $T_{0} \in \Gamma_{0}\left(L,M\right)$ providing $T_{0} \left( \frac{u_{1}}{L} \right)=\frac{u_{2}}{L}$. Then, $T_{0} \left( \frac{u_{1}}{L} \right)\approx T_{0} \left( \infty \right) $ is provided for $\frac{u_{1}}{L} \approx \infty$. Moreover, $\frac{u_{2}}{L}\approx T_{0} \left( \infty \right)$ is obtained.

Secondly, let us demonstrate that $\Gamma_{0}\left(L,M\right)$ permutes the edges. Then, the vertices $\frac{u_{1}}{L}, \frac{u_{2}}{L}$ and $\frac{v_{1}}{L}, \frac{v_{2}}{L}$ are consecutive vertices on the same orbit. Moreover, $\frac{u_{1}}{L}\rightarrow \frac{u_{2}}{L}$ and $\frac{v_{1}}{L}\rightarrow \frac{v_{2}}{L}$ are both edges. There are $T_{0}$ and $S_{0} \in \Gamma_{0}\left(L,M\right)$ providing $S_{0}\left(\infty \right)=\frac{u_{1}}{L}$, $S_{0}\left(\frac{u}{L} \right)=\frac{u_{2}}{L}$ and $T_{0}\left(\infty \right)=\frac{v_{1}}{L}$, $T_{0}\left(\frac{u}{L} \right)=\frac{v_{2}}{L}$. Moreover, $T_{0}S_{0}^{-1}\left(\frac{u_{1}}{L}\right)=\frac{v_{1}}{L}$ and $T_{0} S_{0}^{-1}\left(\frac{u_{2}}{L} \right)=\frac{v_{2}}{L}$. Furthermore, $T_{0} S_{0}^{-1} \in \Gamma_{0}\left(L,M\right)$ is attained.
\end{proof}
\begin{theorem}
$F_{u,L}$ has self paired edges if and only if $u^{2}\equiv \mp 1 \left(\mod L\right)$, where $F_{u,L}$ is a suborbital graph starting with $\infty$.
\end{theorem}
\begin{proof} Assume that $\left( \begin{array}{ccc} 1+aL & b \\ c & 1+dM \end{array} \right)\in \Gamma_{0}\left(L,M\right)$. If $F_{u,L}$ has self paired edges, it is provided that
 $$
\left( \begin{array}{ccc} 1+aL & b \\ c & 1+dM \end{array} \right)\left(\begin{array}{ccc} 1\\0 \end{array}\right)=\left(\begin{array}{ccc} u\\L \end{array}\right), \left( \begin{array}{ccc} 1+aL & b \\ c & 1+dM \end{array} \right)\left(\begin{array}{ccc} u\\L \end{array}\right)=\left(\begin{array}{ccc} 1\\0 \end{array}\right).
$$
Then,
$$
\left( \begin{array}{ccc} 1+aL & b \\ c & 1+dM \end{array} \right)\left( \begin{array}{ccc} 1 & u \\ 0 & L \end{array} \right)=\left( \begin{array}{ccc} u & 1 \\ L & 0 \end{array} \right),
$$
$$
\left( \begin{array}{ccc} 1+aL & u(1+aL)+bL \\ c & cu+L(1+dM) \end{array} \right)=\left( \begin{array}{ccc} u & 1 \\ L & 0 \end{array} \right).
$$
So, $u=1+aL$,$c=L$, $u(1+aL)+bL=1$ and $cu+L(1+dM)=0$ are obtained.$u^{2}+bL=1$ is achieved from $u=1+aL$. So, $u^{2}\equiv 1 \left(\mod L\right)$ is obtained.In addition to, $1+dM=-u$ is found for $c=L$. If values of $1+aL$ and $1+dM$ are substituted in $\left( \begin{array}{ccc} 1+aL & b \\ c & 1+dM \end{array} \right)$, the determinant of this matrix obtained as $-u^{2}+bc=1$. Then, $u^{2}\equiv  -1 \left(\mod L\right)$ is reached.

Inversely, assume that $u^{2}\equiv \mp 1 \left(\mod L\right)$, where $F_{u,L}$ is a suborbital graph starting with $\infty$. Then, $u^{2}= 1+b_{1}L$ and $u^{2}=-1+b_{2}L$ are obtained. For $u^{2}\equiv 1 \left(\mod L\right)$, it can be written from determination,
$$
\left( \begin{array}{ccc} u & b_{1} \\ L & u \end{array} \right)\in \Gamma_{0}\left(L,M\right).
$$
Similarly, for $u^{2}\equiv -1 \left(\mod L\right)$,
$$
\left( \begin{array}{ccc} u & b_{2} \\ L & -u \end{array} \right)\in \Gamma_{0}\left(L,M\right).
$$
\end{proof}

\subsection{Some Conclusions Obtained by $\Gamma_{0}\left(L,M\right)$ on $F_{M,u}$}
Here, we want to achieve some conclusions obtained by $\Gamma_{0}\left(L,M\right)$ and $\widehat{\mathbb{Q}}$. Also, the transitive action of $\Gamma_{0}\left(L,M\right)$ on $\widehat{\mathbb{Q}}$ can easily be demonstrated. To this end, we have to obtain a maximal subgroup from \cite{6}. Moreover, $\Gamma_{0}\left(L,M\right)$ can be selected as maximal subgroup for
 $\Gamma^{0,0}\left(L,M\right)<\Gamma_{0}\left(L,M\right)<\Gamma$ where $\Gamma^{0,0}\left(L,M\right)$ is a subgroup of $\Gamma_{0}\left(L,M\right)$ that stabilizer of $0$.

Since $\Gamma_{0}\left(L,M\right)$ acts transitively on orbit $\left(\begin{array}{ccc} M\\1 \end{array}\right)$, $\left(\Gamma_{0}\left(L,M\right), \widehat{\mathbb{Q}}\right)$ is an imprimitive permutation group.  Unlike the identity and universal, we have to prove that there is a congruence relation. 

Assume that two elements are $g= \left( \begin{array}{ccc} a_{1} & b_{1} \\ c_{1} & d_{1} \end{array} \right)$ and $g^{'}=\left( \begin{array}{ccc} a_{2} & b_{2} \\ c_{2} & d_{2} \end{array} \right)$ in $\Gamma_{0}\left(L,M\right)$, which is similar to the information presented above. So, $g\left(0\right)\approx g^{'}\left(0\right)\Leftrightarrow 0 \approx g^{-1}g^{'}\left(0\right) \Leftrightarrow g^{-1}g^{'} \in \Gamma_{0}\left( L,M \right)$. Hence,
$a_{1} c_{2} -a_{2} c_{1} \equiv 0 \left(\mod L\right)$ and $d_{1} b_{2} -b_{1} d_{2} \equiv 0 \left(\mod M\right)$. Also,   $d_{1}a_{2}\equiv 1 \left(\mod L \right)$ and $a_{1}d_{2}\equiv 1 \left(\mod M \right)$ are found, taking $\Gamma^{0,0}\left(L,M\right)<\Gamma_{0}\left(L,M\right)<\Gamma$,into account. In addition, $d_{1}a_{2}a_{1}\equiv a_{1} \left(\mod L \right)$ is achieved. As a result of this, $d_{1}\equiv a_{1} \left(\mod L \right)$ is reached for  $a_{1}a_{2}\equiv 1 \left(\mod L\right)$. Similarly, $a_{1}d_{2}d_{1}\equiv d_{1}\left(\mod M \right)$   is obtained from $a_{1}d_{2}\equiv 1 \left(\mod M \right)$. Taking $d_{1}d_{2}\equiv 1 \left(\mod M \right)$ into consideration, $a_{1}\equiv d_{1}\left(\mod M \right)$ is reached. So, $g^{-1}g^{'}=T^{0}$ is accepted as $\left( \begin{array}{ccc} 1+aL & b \\ c & 1+dM \end{array} \right)$.

\begin{theorem}
Assume that $\frac{r}{s}$ and $\frac{x}{y}$ are vertices on $F_{M,u}$. There should be an edge from $\frac{r}{s}$ to $\frac{x}{y}$ provided that $x \equiv \mp ur \left(\mod M\right), y \equiv \mp us \left(\mod M\right);ry-sx=M$.
\end{theorem}
\begin{proof}
Firstly, assume that $\frac{r}{s}$ and $\frac{x}{y}$ are vertices on $F_{M,u}$ and there is an edge from $\frac{r}{s}$ to $\frac{x}{y}$. Then, we must obtain an element in $\Gamma_{0}\left(L,M\right)$. So, $T^{0}$ can be used here as an element of $\Gamma_{0}\left(L,M\right)$. Moreover, $T^{0}$ sends $\frac{0}{1}$ to $\frac{r}{s}$ and $\frac{M}{u}$ to $\frac{x}{y}$. Here, it is clear that $r \equiv 0 \left(\mod M\right)$ and $s \equiv 1 \left(\mod M\right)$. Furthermore,
$$
\left( \begin{array}{ccc} 1+aL & b \\ c & 1+dM \end{array} \right)\left( \begin{array}{ccc} 0 & M \\ 1 & u \end{array} \right)=\left( \begin{array}{ccc} r & x \\ s & y \end{array} \right),
$$
$$
\left( \begin{array}{ccc} b & M(1+aL)+bu \\ 1+dM & cM+u(1+dM) \end{array} \right)=\left( \begin{array}{ccc} r & x \\ s & y \end{array} \right).
$$
From the equality of these two matrices, $r=b$ and $s=1+dM$ are obtained. In addition, from these equations, $x=M\left(1+aL\right)+ru$ and $y=cM+us$ are found. So, $x \equiv \mp ur \left(\mod M\right)$ and $y \equiv \mp us \left(\mod M\right)$ are reached. However, $ry-sx=-M$ can clearly be seen  from $\det T^{0}=1$.

On the other hand, assume that $x \equiv \mp ur \left(\mod M\right), y \equiv \mp us \left(\mod M\right)$ ;$ry-sx=M$. Here, we have to achieve the element of Modular group providing an edge from $\frac{r}{s}$ to $\frac{x}{y}$. From the orbit $[0]$, $r \equiv 0 \left(\mod M\right)$ and $s \equiv 1 \left( \mod M\right)$ are found. From $x \equiv \mp ur \left(\mod M\right)$ and $y \equiv \mp us \left(\mod M\right)$, there are $a$ and $c$ in $\mathbb{Z}$ providing $x=ur+aM$ and $y=us+cM$. Then, $ry-sx=M$ is obtained from $x$ and $y$. Moreover,
$$
\left( \begin{array}{ccc} a & r \\ c & s \end{array} \right)\left( \begin{array}{ccc} 0 & M \\ 1 & u \end{array} \right)=\left( \begin{array}{ccc} r & ur+aM \\ s & us+cM \end{array} \right).
$$
From the equation $as-rc=1$, the congruence $as-rc \equiv 1 \left(\mod M\right)$ is obtained. From $r \equiv 0 \left(\mod M\right)$ and $s \equiv 1 \left( \mod M\right)$, $a \equiv 1 \left(\mod L\right)$ is reached. Consequently, $\det \left( \begin{array}{ccc} a & r \\ c & s \end{array} \right)=1$. That is, $\left( \begin{array}{ccc} a & r \\ c & s \end{array} \right)\in \Gamma_{0}\left(L,M\right)$.
\end{proof}
\begin{theorem}
$\Gamma_{0}\left(L,M\right)$ transitively permutes the vertices and the edges of $F_{M,u}$.
\end{theorem}
\begin{proof}
Initially, we will provide a proof related to the vertices. Assume that $\frac{M}{u_{1}},\frac{M}{u_{2}}$  and $\frac{M}{v_{1}},\frac{M}{v_{2}}$ are consecutive vertices and $\frac{M}{u_{1}}\rightarrow\frac{M}{u_{2}}$  and $\frac{M}{v_{1}}\rightarrow\frac{M}{v_{2}}$ are edges in $F_{M,u}$. Thus, we have to indicate that there is an element of a set that provides a certain proof. If $T^{0}$ is taken as $\left( \begin{array}{ccc} a & b \\ c & d \end{array} \right)$,
$$
\left( \begin{array}{ccc} a & b \\ c & d \end{array} \right)\left(\begin{array}{ccc} M\\u_{1} \end{array}\right)=\left(\begin{array}{ccc} M\\u_{2} \end{array}\right)
$$
can be written. So, $aM+bu_{1}=M$ and $cM+du_{1}=u_{2}$ are found. If these equations are solved, taking $\left( M,u_{1} \right)=1$ into account, $b=M$ and $u_{1}=1-a$ for $bu_{1}=M\left(1-a\right)$. Also, they can be found as $a=1-u_{1}, b=M, c=\frac{u_{2}-u_{1} u_{2}-u_{1}}{M}$ and $d=u_{2}+1$ for $\left( M,u_{2} \right)=1$. Then, the relevant matrix can be written as 
$$
\left( \begin{array}{ccc} 1-u_{1} & M \\ \frac{u_{2}-u_{1} u_{2}-u_{1}}{M} & u_{2}+1 \end{array} \right).
$$
Also, the determinant of matrix is $1$. However, if it is taken $1-u_{1}\equiv 1 \left(\mod L\right)$, $u_{2}+1\equiv 1 \left(\mod M\right)$, and $u_{2}\equiv 0 \left(\mod L\right)$, the relevant matrix is found ,which is the element of $\Gamma_{0}\left(L,M\right)$.

On the other hand, we can show that $\Gamma_{0}\left(L,M\right)$ permutes the edges. Then, $T^{0}$ and $S^{0}\in\Gamma_{0}\left(L,M\right)$ provide $S^{0}\left( 0 \right)=\frac{M}{u_{1}},S^{0} \left(\frac{M}{u}\right) =\frac{M}{u_{2}}$  and $T^{0}\left( 0 \right)=\frac{M}{v_{1}},T^{0} \left(\frac{M}{v}\right) =\frac{M}{v_{2}}$. So, $T^{0}\left(S^{0}\right)^{-1} \left( \frac{M}{u_{1}} \right)=\frac{M}{v_{1}}$ and $T^{0}\left(S^{0}\right)^{-1} \left( \frac{M}{u_{2}} \right)=\frac{M}{v_{2}}$ are obtained. As a result, $T^{0}\left(S^{0}\right)^{-1}\in \Gamma_{0}\left(L,M\right)$ permutes the edges.
\end{proof} 
\begin{theorem}
$F_{M,u}$ has self paired edges if and only if $u^{2}\equiv \mp 1 \left(\mod M\right)$, where $F_{M,u}$ is a suborbital graph starting with $0$.
\end{theorem}
\begin{proof} Assume that $F_{M,u}$ has self paired edges. That is, $0\longrightarrow\frac{M}{u}\longrightarrow 0$. In other words, the element of $\Gamma_{0}\left(L,M\right)$ provides,
$$
\left( \begin{array}{ccc} 1+aL & b \\ c & 1+dM \end{array} \right)\left( \begin{array}{ccc} 0 & M \\ 1 & u \end{array} \right)=\left( \begin{array}{ccc} M & 0 \\ u & 1 \end{array} \right),
$$ 
$$
\left( \begin{array}{ccc} b & M\left(1+aL\right)+bu \\ 1+dM & cM+u\left(1+dM\right) \end{array} \right)=\left( \begin{array}{ccc} M & 0 \\ u & 1 \end{array} \right).
$$ 
Then, $b=M$, $1+dM=u$, $M\left(1+aL\right)+bu=0$ and $cM+u\left(1+dM\right)=1$ are obtained. So, $-u=1+aL$ are achieved. Moreover, $u^{2}\equiv 1 \left(\mod M\right)$ is reached. If the values of $a$, $b$, and, $d$ are substituted in the element of $\Gamma_{0}\left(L,M\right)$, 
$$
\left( \begin{array}{ccc} -u & M \\ c & u \end{array} \right)
$$
is achieved. The determination of this matrix, $-u^{2}-cM=1$ is found. From here, $u^{2}\equiv -1 \left(\mod M\right)$ is provided.

Conversely, the proof of this theorem can be provided easily.
\end{proof}
\begin{theorem}
Suborbital graph paired with $F_{M,u}$ is $F_{-M,\overline{u}}$ for $u\overline{u}\equiv 1 \left(\mod M\right).$
\end{theorem}
\begin{proof}
The paired suborbital graphs have the same edge in opposite directions. In other words, $\frac{r}{s}\longrightarrow \frac{x}{y}\in F_{M,u}$ and $\frac{x}{y}\longrightarrow \frac{r}{s}\in F_{-M,\overline{u}}.$ $O\left(\frac{0}{1},\frac{M}{u}\right)$ and $O\left(\frac{0}{1},\frac{x}{y}\right)$ are orbitals. Since these orbitals are paired, $T^{0}\left(\frac{0}{1}\right)=\frac{M}{u}$ and $T^{0}\left(\frac{x}{y}\right)=\frac{0}{1}$ are provided. So, matrix multiplication
$$
-\left( \begin{array}{ccc} 1+aL & b \\ c & 1+dM \end{array} \right)\left( \begin{array}{ccc} 0 & x \\ 1 & y \end{array} \right)=\left( \begin{array}{ccc} M & 0 \\ u & 1 \end{array} \right)
$$
can be written. Then,
$$
\left( \begin{array}{ccc} -b & -x\left(1+aL\right)-by \\ -\left(1+dM\right) & -xc-y\left(1+dM\right)\end{array} \right)=\left( \begin{array}{ccc} M & 0 \\ u & 1 \end{array} \right).
$$ 
From the equation of two matrices, $-b=M$ and $u=-\left(1+dM\right)$ are obtained. If these values are substituted in $-xc-y\left(1+dM\right)=1$, $uy\equiv 1 \left(\mod M\right)$ is found. In addition, $ry-sx=M$ is achieved for the edge $\frac{r}{s}\longrightarrow\frac{x}{y}$. Similarly, $ry-sx=-M$ is found for the edge $\frac{x}{y}\longrightarrow\frac{r}{s}$. However, the edge condition is presented as $x \equiv ur \left(\mod M\right), y \equiv us \left( mod M\right);ry-sx=M$ for an edge from $\frac{r}{s}$ to $\frac{x}{y}$ in Theorem 7. From here, $\overline{u}x \equiv u\overline{u}r \left(\mod M\right)$ is obtained. If it is taken $u\overline{u}\equiv 1 \left(\mod M\right)$, $\overline{u}x \equiv r \left(\mod M\right)$ is found. That is, $r \equiv \overline{u}x \left(\mod M\right)$. Similarly, $y\equiv us \left(\mod M\right)$ is given in Theorem 7. $\overline{u}y\equiv u\overline{u}s \left(\mod M\right)$ is obtained. Then,  $\overline{u}y\equiv s \left(\mod M\right)$. So, $s\equiv \overline{u}y \left(\mod M\right)$.
\end{proof}
\section{Conclusion}\label{sec3}
In this paper, the modular group $\Gamma_{0} \left(L,M\right)$ based on two numbers is presented and the suborbital graphs formed by the action of this modular group are analyzed. Since $\Gamma_{0} \left(L,M\right)<\Gamma_{0}\left(L\right)$,  $\Gamma_{0}\left(L,M\right)$ generates orbital starting with $\infty$. In addition, $\Gamma_{0} \left(L,M\right)$ generates orbital starting with 0 for $\Gamma_{0} \left(L,M\right)<\Gamma^{0}\left(M\right)$. It can be seen that congruence modular group $\Gamma_{0} \left(L,M\right)$ forms different orbitals according to $\infty$ and $0$. Here, it can be given $\frac{u_{2}-u_{1}u_{2}-u_{1}}{M}\equiv 0 \left(\mod L \right)$ for $T^{0}\in\Gamma_{0}\left(L,M\right)$ . Then, $u_{2}-u_{1}u_{2}-u_{1}=kML$ is obtained. From here, $\frac{u_{2}-u_{1}u_{2}-u_{1}}{M}\equiv 0 \left(\mod M \right)$ is achieved. Similarly, results are obtained for  $T_{0}\in\Gamma_{0}\left(L,M\right)$. 
Thus, by studying a different congruence modular subgroup, this paper gives various aspects of the same structure from different perspectives in mathematics and adapts it to different fields such as algebraic geometry, number theory, differential geometry, topology, and physics.

\label{son}
\end{document}